\documentclass[11pt,a4paper]{article}
\usepackage[T1]{fontenc}
\usepackage{lmodern}
\usepackage{amsmath,amssymb,amsthm,mathtools}
\usepackage{geometry}
\usepackage{enumitem,booktabs,array,microtype}
\usepackage{xcolor}
\usepackage[colorlinks=true,linkcolor=blue!45!black,citecolor=blue!45!black,urlcolor=blue!45!black]{hyperref}
\numberwithin{equation}{section}

\newtheorem{theorem}{Theorem}[section]
\newtheorem{proposition}[theorem]{Proposition}
\newtheorem{lemma}[theorem]{Lemma}
\newtheorem{corollary}[theorem]{Corollary}
\theoremstyle{definition}
\newtheorem{definition}[theorem]{Definition}

\theoremstyle{remark}

\DeclareMathOperator{\diver}{div}
\DeclareMathOperator{\Vol}{Vol}
\DeclareMathOperator{\tr}{tr}
\newcommand{\R}{\mathbb R}
\newcommand{\dd}{\,d}
\newcommand{\nbar}{\bar n}

\newcommand{\BH}{\mathrm{BH}}
\newcommand{\HT}{\mathrm{HT}}

\title{Hypersurface and Foliation Isoparametric Transfer\\for Generalized $(\alpha,\beta)$-Metrics}
\author{Yali Chen\\[0.55em]
\small School of Mathematics and Statistics, Anhui Normal University,\\
\small Wuhu, Anhui 241000, China}
\date{}

\begin{document}
\maketitle

\begin{abstract}
Let $F=\alpha\Phi(q,s)$ be a regular generalized $(\alpha,\beta)$-metric,
where $q=\|\beta\|_\alpha^2$ and $s=\beta/\alpha$, endowed with the
Busemann--Hausdorff or Holmes--Thompson volume.  We derive explicit comparison
formulas for the induced hypersurface volume and the associated volume mean
curvature and use them to study two transfer problems suggested by the Randers
case.  The first asks whether the same hypersurface is isoparametric for
$\alpha$ and $F$; the second asks whether the same regular foliation is
isoparametric for both metrics.  Exact local and global criteria are obtained.
Their Randers specializations recover the corresponding hypersurface transfer
and the condition $df(W)=\varphi(f)$.  We also prove an all-normal
induced-volume rigidity theorem and construct a complete non-Randers example
realizing both transfer mechanisms.
\end{abstract}


\section{Introduction}

The geometry of hypersurfaces in a Finsler manifold depends essentially on the
choice of volume.  Unlike the Riemannian case, several natural hypersurface
volumes and several notions of mean curvature occur.  In this paper the ambient
measure is the Busemann--Hausdorff (BH) or Holmes--Thompson (HT) volume, the
hypersurface measure is obtained by contraction with a chosen Finsler unit
normal, and the corresponding mean curvature is the one arising from the first
variation of this induced volume.  This point of view is standard in the study
of Finsler submanifolds and is particularly suitable for isoparametric
geometry; see, for example, \cite{HeShen2006,ChenHeQian2025}.

The class of $(\alpha,\beta)$-metrics provides one of the main sources of
explicit non-Riemannian Finsler structures.  Besides Randers metrics, it
contains many important examples whose geometric quantities can still be
computed in terms of a Riemannian metric $\alpha$ and a one-form $\beta$.
Projectively flat and curvature-related problems for ordinary
$(\alpha,\beta)$-metrics have been studied extensively; see, for instance,
\cite{Shen2009}.  A further extension is the class of general or generalized
$(\alpha,\beta)$-metrics introduced by Yu and Zhu,
\[
 F=\alpha\Phi(q,s),\qquad q=\|\beta\|_\alpha^2,\qquad
 s=\frac{\beta}{\alpha},
\]
where the additional $q$-dependence gives considerably more flexibility than
in the ordinary case; see \cite{YuZhu2011,YuZhu2015}.  Most work on this class
concerns intrinsic curvature and projective properties.  By comparison, the
extrinsic geometry of hypersurfaces in general $(\alpha,\beta)$-manifolds is
less developed, although even in the ordinary $(\alpha,\beta)$ case the
induced volume and mean curvature already contain genuinely Finslerian
correction terms.

Isoparametric hypersurfaces are a natural setting in which these correction
terms become visible.  Isoparametric functions and hypersurfaces have been
studied in Minkowski spaces and in several other Finsler settings; see
\cite{HeYinShen2016,HeHuangDong2022}.  The Randers case is especially
instructive.  Let $F$ be a Randers space form with navigation datum $(h,W)$.
He--Dong--Yin proved that $(N,F,d\mu_{\BH})$ and $(N,h)$ have the same
isoparametric hypersurfaces and hence the same classification list, although
their natural parallel families need not coincide \cite{HeDongYin2020}.  They
also proved that an $h$-isoparametric function $f$ is $F$-isoparametric if and
only if
\begin{equation}\label{eq:intro-randers-II}
 df(W)=\varphi(f)
\end{equation}
for a one-variable function $\varphi$.

These two Randers results lead to two different transfer problems.  In the
\emph{same-hypersurface transfer problem}, one fixes a cooriented hypersurface
$M\subset N$ and asks whether the same subset is isoparametric for both
metrics; the two normal families are allowed to differ.  In the
\emph{same-foliation transfer problem}, one fixes a regular foliation
$M_t=f^{-1}(t)$ and asks whether this same foliation is isoparametric for both
$\alpha$ and $F$.  The latter is therefore a statement about an entire family
of leaves rather than a single hypersurface.

Our aim is to extend these two Randers transfer mechanisms to regular
generalized $(\alpha,\beta)$-metrics.  Three difficulties appear.  First, the
simple Randers navigation identity between the two unit normals has no direct
analogue for a general profile $\Phi(q,s)$.  Second, in the same-hypersurface
problem the actual $F$-normal tube of $M$ need not agree with its
$\alpha$-normal tube.  Third, the mean curvature relevant for the BH or HT
Laplacian is the volume mean curvature associated with the induced
hypersurface measure, rather than only the trace of the anisotropic shape
operator.  We overcome these difficulties by working first with the dual norm
and the Legendre transformation.  This yields an explicit comparison of the
unit normals and the induced hypersurface volumes, and the first variation then
gives a pointwise comparison formula for the two volume mean curvatures.  The
same formula is the basic tool for both transfer problems.

The first main result concerns a prescribed hypersurface and its actual
$F$-normal family.  In global form it may be stated as follows; the notation in
the criterion is introduced in Sections~\ref{sec:prelim} and \ref{sec:mean}.

\begin{theorem}[Global same-hypersurface transfer]\label{thm:intro-global-I}
Let $M$ be a connected cooriented $\alpha$-isoparametric hypersurface.  Suppose
that $M=g^{-1}(t_0)$ is a connected regular level of a globally defined
$F$-transnormal function $g$, and write $M_t^F=g^{-1}(t)$.  Then $g$ is
$F$-isoparametric if and only if every connected regular level satisfies
\begin{equation}\label{eq:intro-global-I}
 d_{M_t^F}
 \left[
 \frac{\Lambda(q_t,r_t)H_{\alpha,t}-\mathcal Q_{M_t^F}}
 {\rho(q_t)}
 \right]=0.
\end{equation}
In particular, the condition is imposed on the actual $F$-transnormal leaves;
no equality with the $\alpha$-parallel family is assumed.
\end{theorem}

The second main result concerns one prescribed global foliation.

\begin{theorem}[Global same-foliation transfer]\label{thm:intro-global-II}
Let $f$ be a non-constant $C^1$ function on $N$, smooth on
$N_f=\{df\ne0\}$, whose connected regular levels form an
$\alpha$-isoparametric foliation.  Then the same function $f$ is
$F$-isoparametric if and only if, on every regular leaf $M_t=f^{-1}(t)$,
\begin{equation}\label{eq:intro-global-II}
 d_{M_t}\psi(q,r)=0,
 \qquad
 d_{M_t}\mathcal C_{M_t}=0.
\end{equation}
When these conditions hold, $\psi(q,r)=\psi(f)$ and
$\mathcal C_{M_t}=C(f)$ on the regular set, and
\begin{equation}\label{eq:intro-aF}
 F^*(df)=a_\alpha(f)\psi(f).
\end{equation}
\end{theorem}

For Randers metrics the first theorem recovers the same-hypersurface mechanism
behind the classification theorem of He--Dong--Yin, while the second reduces
to \eqref{eq:intro-randers-II}.  We also show that an induced-volume identity
imposed in every normal direction is rigid and forces the Finsler norm to be
conformal to $\alpha$.  The final section gives a smooth forward and backward
complete generalized $(\alpha,\beta)$-metric which is non-Randers and
non-locally Minkowski and for which $f(x)=|x|^2$ realizes both transfer
mechanisms.

\section{Preliminaries}\label{sec:prelim}

\subsection{Generalized $(\alpha,\beta)$-metrics, duality, and BH/HT volumes}

Let $(N^n,\alpha)$ be an oriented Riemannian manifold, $n\ge2$, and let $\beta$
be a smooth one-form.  Write
\begin{equation}\label{eq:prelim-F}
 F(x,y)=\alpha(x,y)\Phi(q,s),\qquad
 q=\|\beta\|_\alpha^2,\qquad
 s=\frac{\beta(y)}{\alpha(y)},\qquad
 B=\beta^{\sharp_\alpha}.
\end{equation}
We assume throughout that $F$ is a smooth strongly convex Finsler metric on
$TN\setminus0$.  On the realized range $|s|\le\sqrt q$ it is convenient to put
\begin{equation}\label{eq:convexity}
 P=\Phi-s\Phi_s,
 \qquad
 D_0=\Phi-s\Phi_s+(q-s^2)\Phi_{ss},
\end{equation}
and to assume the usual regularity inequalities
\begin{equation}\label{eq:convexity-ineq}
 \Phi>0,\qquad P>0,\qquad D_0>0.
\end{equation}
No singular Kropina or general conic metric is included in the statements below.

For a Finsler metric $F$, the Busemann--Hausdorff density in local coordinates is
\begin{equation}\label{eq:BH-def}
 \sigma_{\BH}(x)=
 \frac{\Vol(B^n)}{\Vol\{y\in\R^n:F(x,y)<1\}},
 \qquad
 dV_{\BH}=\sigma_{\BH}(x)\,dx^1\wedge\cdots\wedge dx^n.
\end{equation}
The Holmes--Thompson density may be written as the normalized symplectic volume
of the unit co-ball, equivalently
\begin{equation}\label{eq:HT-def}
 \sigma_{\HT}(x)=
 \frac{1}{\Vol(B^n)}
 \int_{\{F(x,y)<1\}}\det(g_{ij}(x,y))\,dy^1\cdots dy^n,
\end{equation}
where $g_{ij}=\frac12(F^2)_{y^iy^j}$ is the fundamental tensor.

For the metric \eqref{eq:prelim-F}, both densities are multiples of the
Riemannian density.  We write uniformly
\begin{equation}\label{eq:rho-def}
 dV_F=\rho(q)dV_\alpha,
\end{equation}
where $dV_F$ denotes either BH or HT volume.  Put
\begin{equation*}
 I_n=\int_0^\pi\sin^{n-2}\theta\,d\theta,
 \qquad s=\sqrt q\cos\theta.
\end{equation*}
A polar integration in an $\alpha$-orthonormal frame gives
\begin{equation}\label{eq:rhoBH}
 \rho_{\BH}(q)=
 \frac{I_n}{\displaystyle
 \int_0^\pi \Phi(q,\sqrt q\cos\theta)^{-n}
 \sin^{n-2}\theta\,d\theta},
\end{equation}
and
\begin{equation}\label{eq:rhoHT}
 \rho_{\HT}(q)=
 \frac1{I_n}\int_0^\pi
 \Phi P^{n-2}D_0\sin^{n-2}\theta\,d\theta.
\end{equation}
Indeed,
\begin{equation}\label{eq:detg}
 \frac{\det(g_{ij})}{\det(a_{ij})}
 =\Phi^{n+1}P^{n-2}D_0,
\end{equation}
and radial integration over $F(y)<1$ contributes the factor
$\Phi^{-n}/n$.  In particular, for either choice of volume, $\rho$ depends only
on the base-point scalar $q$.  If $F=A(x)F_0$ with $A>0$, then
\begin{equation}\label{eq:rho-scale}
 \rho_{AF_0}=A^n\rho_{F_0}.
\end{equation}

The Legendre transformation is
\begin{equation}\label{eq:Legendre}
 \mathcal L_F(y)=g_y(y,\cdot)=F(y)d_yF,
 \qquad
 \mathcal L_F^{-1}(\xi)=F^*(\xi)d_\xi F^*(\xi).
\end{equation}
It is positively one-homogeneous but generally not odd.  We write the dual norm
as
\begin{equation}\label{eq:dual-profile}
 F^*(\xi)=\alpha^*(\xi)\psi(q,r),
 \qquad
 r=\frac{\xi(B)}{\alpha^*(\xi)}.
\end{equation}
The primal variable $s$ and the dual variable $r$ should not be identified
without the Legendre conversion derived in Section~\ref{sec:mean}.

\subsection{Finsler submanifolds, normal bundles, and mean curvature}

Let $\iota:M^m\to(N^n,F)$ be an immersion.  The conormal bundle is the vector
bundle
\begin{equation}\label{eq:conormal-bundle}
 N^*M=\{(x,\xi):x\in M,\ \xi(X)=0\ \text{for all }X\in T_xM\}.
\end{equation}
The Finsler normal bundle is defined by
\begin{equation}\label{eq:normal-bundle}
 NM=\{(x,n):x\in M,\ \mathcal L_F(n)(X)=0\ \text{for all }X\in T_xM\}.
\end{equation}
Unlike $N^*M$, the set $NM$ is generally not a vector bundle.  Away from the
zero section the Legendre map gives a diffeomorphism between the two normal
objects.  The unit normal bundle is
\begin{equation}\label{eq:unit-normal-bundle}
 N^0M=\{n\in NM:F(n)=1\}.
\end{equation}
For an embedded cooriented hypersurface we fix one of the two unit normal
sections and denote it by $n$.

For $n\in N^0M$, the pullback of the fundamental tensor $g_n$ gives a
Riemannian metric on $T_xM$.  This direction-dependent family
$(M,\iota^*g_n)$ is often called the anisotropic submanifold associated with
$n$; see \cite{HeDongYin2020}.  Let $D^n$ denote the Chern covariant derivative
with reference vector $n$.  The anisotropic shape operator is
\begin{equation}\label{eq:shape-operator}
 A_n(X)=-\bigl(D_X^n n\bigr)^{\top}_{g_n}.
\end{equation}
It is self-adjoint with respect to $\iota^*g_n$.  Its eigenvalues
$\lambda_1,\ldots,\lambda_m$ are the anisotropic principal curvatures and
\begin{equation}\label{eq:anisotropic-H}
 \widehat H_n=\tr A_n=\sum_{a=1}^m\lambda_a
\end{equation}
is the unnormalized anisotropic mean curvature.  These notions are important in
the Randers classification, where the shape operators can be compared directly.
They should, however, be distinguished from the volume mean curvature used in
this paper.

Fix a smooth ambient volume $dV_F$.  If $M$ is a cooriented hypersurface with
unit normal $n$, we define its contracted hypersurface volume by
\begin{equation}\label{eq:induced-volume-def}
 dV_n^F=(\iota_n dV_F)|_M.
\end{equation}
For a variation with vector field $X$, the corresponding volume mean curvature
$H_F$ is characterized by
\begin{equation}\label{eq:first-variation-def}
 \delta A_F(X)=-\int_M H_F\,\nu(X)\,dV_n^F,
 \qquad \nu=\mathcal L_F(n).
\end{equation}
Equivalently, for the chosen orientation,
\begin{equation}\label{eq:HF-div}
 H_F=-\diver_{dV_F}n.
\end{equation}
The Riemannian convention used throughout is similarly
$H_\alpha=-\diver_\alpha\bar n$, with no factor $1/(n-1)$.
The contracted measure \eqref{eq:induced-volume-def} need not coincide with the
intrinsic BH or HT measure of the restricted norm $F|_{TM}$; the first variation
in this paper always refers to \eqref{eq:induced-volume-def}.

\subsection{Transnormal and isoparametric functions}\label{subsec:isoparametric-prelim}

Let $f$ be a non-constant $C^1$ function on a Finsler manifold $(N,F,dV_F)$
and assume that $f$ is smooth on
\[
 N_f=\{x\in N:df(x)\ne0\}.
\]
Set $J=f(N_f)$.  Following the standard definition in Finsler isoparametric
geometry, $f$ is called \emph{$F$-isoparametric} if there are a smooth function
$a_F$ and a continuous function $b_F$ on $J$ such that
\begin{equation}\label{eq:F-isoparametric-definition}
 F^*(df)=a_F(f),
 \qquad
 \Delta_Ff=b_F(f)
\end{equation}
on $N_f$.  Equivalently, $F(\nabla^Ff)=a_F(f)$, where
\begin{equation}\label{eq:grad-lap}
 \nabla^Ff=\mathcal L_F^{-1}(df),
 \qquad
 \Delta_Ff=\diver_{dV_F}(\nabla^Ff).
\end{equation}
A function satisfying only the first equation in
\eqref{eq:F-isoparametric-definition} is called \emph{$F$-transnormal}.  The
analogous notions for $\alpha$ use $\alpha^*$ and $\Delta_\alpha$.

The regular level hypersurfaces $M_t=f^{-1}(t)$ form an isoparametric family.
An embedded hypersurface $M$ is called locally $F$-isoparametric if every point
of $M$ has a neighborhood in which $M$ is a regular level of an
$F$-isoparametric function.  We call an isoparametric function global when it
is defined on the whole ambient manifold in the above $C^1$ sense.  No
additional regularity or submanifold structure is required of the critical
level sets in the present paper.

If $f$ is $F$-transnormal, then on $N_f$
\begin{equation}\label{eq:unitnormal-f}
 n=\frac{\nabla^Ff}{a_F(f)},
 \qquad
 n(f)=a_F(f).
\end{equation}
Consequently,
\begin{align}
 \Delta_Ff
 &=\diver_{dV_F}\bigl(a_F(f)n\bigr)\notag\\
 &=a_F(f)a_F'(f)+a_F(f)\diver_{dV_F}n\notag\\
 &=a_F(f)a_F'(f)-a_F(f)H_F.
\label{eq:lap-H-relation}
\end{align}
Hence, on the regular set of an isoparametric function,
\begin{equation}\label{eq:H-from-lap}
 H_F=a_F'(f)-\frac{b_F(f)}{a_F(f)}.
\end{equation}
We shall repeatedly use the following standard consequence.

\begin{lemma}[Transnormal--CMC criterion]\label{lem:transnormal-CMC}
Let $f$ be an $F$-transnormal function with connected regular levels.  Then
$f$ is $F$-isoparametric if and only if every regular level has constant volume
mean curvature with respect to the normal selected by $df$.
\end{lemma}

If $d_F$ is a local forward normal-distance coordinate to a cooriented
hypersurface, then
\begin{equation}\label{eq:distance-transnormal}
 F^*(dd_F)=1
\end{equation}
on its regular tube.  In a nonreversible metric this is a one-sided statement;
the opposite side is treated separately or by the reverse metric.

Isoparametricity is unchanged by increasing regular reparametrization.  Indeed,
if $\widetilde f=\chi(f)$ with $\chi'>0$, then
\begin{equation}\label{eq:reparam-grad}
 \nabla^F\widetilde f=\chi'(f)\nabla^Ff
\end{equation}
and
\begin{equation}\label{eq:reparam-lap}
 \Delta_F\widetilde f
 =\chi''(f)F^*(df)^2+\chi'(f)\Delta_Ff.
\end{equation}
Thus, in the second transfer problem, the essential object is the regular
foliation rather than a particular numerical parametrization of its leaves.

\section{Hypersurface volume and volume mean-curvature comparison}\label{sec:mean}

\subsection{Primal--dual conversion}

The first step is to express the dual profile $\psi$ in terms of the primal
profile $\Phi$.  Put
\begin{equation}\label{eq:E-and-R}
 E(q,s)=\sqrt{\Phi(q,s)^2+(q-s^2)\Phi_s(q,s)^2},
 \qquad
 R(q,s)=\frac{s\Phi+(q-s^2)\Phi_s}{E(q,s)}.
\end{equation}

\begin{lemma}[Primal--dual conversion]\label{lem:primal-dual}
For the generalized $(\alpha,\beta)$-metric \eqref{eq:prelim-F},
\begin{equation}\label{eq:psi-conversion}
 \psi(q,R(q,s))=E(q,s)^{-1}.
\end{equation}
Moreover,
\begin{equation}\label{eq:E-R-derivatives}
 E_s=\frac{\Phi_sD_0}{E},
 \qquad
 R_s=\frac{\Phi P D_0}{E^3}>0,
\end{equation}
and hence
\begin{equation}\label{eq:psi-r-formula}
 \boxed{\psi_r(q,R(q,s))=-\frac{\Phi_s}{\Phi P}.}
\end{equation}
For $q>0$, the map $s\mapsto R(q,s)$ maps
$[-\sqrt q,\sqrt q]$ monotonically onto itself.
\end{lemma}

\begin{proof}
At a fixed base point let
\begin{equation*}
 \ell_i=\frac{a_{ij}y^j}{\alpha}.
\end{equation*}
Since
\begin{equation*}
 \frac{\partial s}{\partial y^i}
 =\frac{b_i-s\ell_i}{\alpha},
\end{equation*}
we obtain
\begin{align}
 F_{y^i}
 &=\frac{\partial}{\partial y^i}(\alpha\Phi)\notag\\
 &=\Phi\ell_i+\Phi_s(b_i-s\ell_i)\notag\\
 &=P\ell_i+\Phi_s b_i.
\label{eq:Fyi}
\end{align}
Let $p=d_yF$.  Using
$a^{ij}\ell_i\ell_j=1$, $a^{ij}\ell_i b_j=s$, and
$a^{ij}b_i b_j=q$, one finds
\begin{align}
 \alpha^*(p)^2
 &=a^{ij}(P\ell_i+\Phi_s b_i)(P\ell_j+\Phi_s b_j)\notag\\
 &=P^2+2sP\Phi_s+q\Phi_s^2\notag\\
 &=\Phi^2+(q-s^2)\Phi_s^2
 =E^2,
\label{eq:alpha-dual-p}
\end{align}
and
\begin{align}
 p(B)
 &=sP+q\Phi_s\notag\\
 &=s\Phi+(q-s^2)\Phi_s.
\label{eq:pB}
\end{align}
Therefore the dual angular variable of $p$ is exactly $R(q,s)$.  By positive
homogeneity,
\begin{equation*}
 F^*(d_yF)=1,
\end{equation*}
so
\begin{equation*}
 1=F^*(p)=\alpha^*(p)\psi(q,R)=E\psi(q,R),
\end{equation*}
which proves \eqref{eq:psi-conversion}.

Differentiate $E^2$ with respect to $s$:
\begin{align*}
 EE_s
 &=\Phi\Phi_s-s\Phi_s^2+(q-s^2)\Phi_s\Phi_{ss}\\
 &=\Phi_s\bigl[\Phi-s\Phi_s+(q-s^2)\Phi_{ss}\bigr]
 =\Phi_sD_0.
\end{align*}
Thus $E_s=\Phi_sD_0/E$.  If
$N=s\Phi+(q-s^2)\Phi_s$, then $N_s=D_0$ and
\begin{equation*}
 E^2-N\Phi_s=\Phi(\Phi-s\Phi_s)=\Phi P.
\end{equation*}
Hence
\begin{equation*}
 R_s=\frac{N_sE-NE_s}{E^2}
 =\frac{D_0(E^2-N\Phi_s)}{E^3}
 =\frac{\Phi P D_0}{E^3}>0.
\end{equation*}
At $s=\pm\sqrt q$ one has $R=\pm\sqrt q$, so the entire dual interval is
covered.  Finally, differentiating
$\psi(q,R(q,s))=E^{-1}$ and dividing by $R_s$ gives
\eqref{eq:psi-r-formula}.
\end{proof}

\subsection{Normal vector and induced hypersurface volume}

Let $M\subset N$ be a smooth cooriented hypersurface.  Denote by $n$ the chosen
$F$-unit normal and by
\begin{equation*}
 \nu=\mathcal L_F(n)
\end{equation*}
the corresponding unit conormal.  Let $\bar n$ be the $\alpha$-unit normal on
the same side and
\begin{equation*}
 \bar\nu=\mathcal L_\alpha(\bar n)=a(\bar n,\cdot).
\end{equation*}
Put
\begin{equation}\label{eq:r-Btop}
 r=\beta(\bar n)=a(B,\bar n),
 \qquad
 B^\top=B-r\bar n.
\end{equation}
Then $B^\top$ is tangent to $M$.

\begin{lemma}[Normal and induced-volume comparison]\label{lem:normal-volume}
For every cooriented hypersurface,
\begin{equation}\label{eq:conormal-comparison}
 \boxed{\nu=\frac{\bar\nu}{\psi(q,r)},}
\end{equation}
\begin{equation}\label{eq:normal-comparison}
 \boxed{n=\psi(q,r)\bar n+\psi_r(q,r)B^\top,}
\end{equation}
and
\begin{equation}\label{eq:induced-volume-comparison}
 \boxed{dV_n^F=\rho(q)\psi(q,r)dV_{\bar n}^\alpha.}
\end{equation}
If
\begin{equation}\label{eq:Lambda-def}
 \Lambda(q,r)=\rho(q)\psi(q,r),
\end{equation}
then $dV_n^F=\Lambda dV_{\bar n}^\alpha$.  Moreover, for every variation field
$X$ along $M$,
\begin{equation}\label{eq:flux-comparison}
 \nu(X)dV_n^F=\rho(q)\bar\nu(X)dV_{\bar n}^\alpha.
\end{equation}
\end{lemma}

\begin{proof}
The conormals $\nu$ and $\bar\nu$ annihilate the same tangent hyperplane and
have the same coorientation, so $\nu=c\bar\nu$ for some $c>0$.  Since
$F^*(\nu)=1$,
\begin{equation*}
 1=cF^*(\bar\nu)
 =c\alpha^*(\bar\nu)\psi(q,r)
 =c\psi(q,r),
\end{equation*}
which proves \eqref{eq:conormal-comparison}.

For a positively one-homogeneous norm,
\begin{equation*}
 \mathcal L_F^{-1}(\xi)=F^*(\xi)d_\xi F^*(\xi).
\end{equation*}
At the unit conormal $\nu$, zero-homogeneity of $dF^*$ gives
\begin{equation*}
 n=d_\xi F^*(\bar\nu).
\end{equation*}
At $\bar\nu$, $d_\xi\alpha^*=\bar n$.  Since
$r(\xi)=\xi(B)/\alpha^*(\xi)$,
\begin{equation*}
 d_\xi r(\bar\nu)=B-r\bar n=B^\top.
\end{equation*}
Differentiating \eqref{eq:dual-profile} therefore gives
\eqref{eq:normal-comparison}.

Let $e_1,\ldots,e_{n-1}$ be a positively oriented basis of $T_xM$.  Using
$dV_F=\rho dV_\alpha$ and \eqref{eq:normal-comparison},
\begin{align*}
 dV_n^F(e_1,\ldots,e_{n-1})
 &=\rho dV_\alpha(\psi\bar n+\psi_rB^\top,e_1,\ldots,e_{n-1})\\
 &=\rho\psi\,dV_\alpha(\bar n,e_1,\ldots,e_{n-1}),
\end{align*}
because the term containing $B^\top$ vanishes.  This proves
\eqref{eq:induced-volume-comparison}.  Finally,
\eqref{eq:conormal-comparison} and \eqref{eq:induced-volume-comparison} give
\begin{equation*}
 \nu(X)dV_n^F
 =\frac{\bar\nu(X)}{\psi}\rho\psi dV_{\bar n}^\alpha
 =\rho\bar\nu(X)dV_{\bar n}^\alpha.
\end{equation*}
\end{proof}

\subsection{The general volume mean-curvature formula}

We now derive the main first-variation formula.  Define
\begin{equation}\label{eq:mean-curvatures}
 H_\alpha=-\diver_\alpha\bar n,
 \qquad
 H_F=-\diver_{dV_F}n.
\end{equation}

\begin{proposition}[General volume mean-curvature comparison]\label{prop:general-H}
For every smooth cooriented hypersurface $M$,
\begin{equation}\label{eq:general-H}
 \boxed{\rho H_F=\Lambda H_\alpha-\mathcal Q_M,}
\end{equation}
where
\begin{equation}\label{eq:general-Q}
 \boxed{
 \mathcal Q_M=
 \Lambda_q\bar n(q)
 +\Lambda_r(\nabla^\alpha_{\bar n}\beta)(\bar n)
 +\diver_M(\Lambda_rB^\top).
 }
\end{equation}
No leafwise constancy assumption on $q$, $r$, or the divergence term is used.
\end{proposition}

\begin{proof}
Take an $\alpha$-normal variation field
\begin{equation*}
 X=\eta\bar n,
\end{equation*}
compactly supported if $M$ is noncompact.  Since $q$ is a scalar function on
$N$,
\begin{equation}\label{eq:qdot}
 \dot q=dq(X)=\eta\bar n(q).
\end{equation}
For
$r=\beta(\bar n)$ the covariant product rule gives
\begin{equation}\label{eq:rdot-start}
 \dot r=(\nabla_X^\alpha\beta)(\bar n)
 +\beta\left(\frac{D^\alpha\bar n}{d\varepsilon}\right).
\end{equation}
We compute the variation of the Riemannian unit normal.  Let
$e_a(\varepsilon)$ be tangent fields along the variation.  Differentiating
$\langle\bar n_\varepsilon,e_a(\varepsilon)\rangle_\alpha=0$ and using the
commutation of the variation parameters gives
\begin{equation*}
 \left\langle\frac{D^\alpha\bar n}{d\varepsilon},e_a\right\rangle_\alpha
 =-e_a(\eta).
\end{equation*}
Differentiating $|\bar n_\varepsilon|_\alpha^2=1$ shows that the normal component
vanishes.  Hence
\begin{equation}\label{eq:normal-variation}
 \frac{D^\alpha\bar n}{d\varepsilon}=-\nabla^M\eta.
\end{equation}
Substitution in \eqref{eq:rdot-start} gives
\begin{equation}\label{eq:rdot}
 \dot r
 =\eta(\nabla^\alpha_{\bar n}\beta)(\bar n)
 -\langle B^\top,\nabla^M\eta\rangle_\alpha.
\end{equation}
Therefore
\begin{align}
 \dot\Lambda
 &=\Lambda_q\dot q+\Lambda_r\dot r\notag\\
 &=\Bigl[\Lambda_q\bar n(q)
 +\Lambda_r(\nabla^\alpha_{\bar n}\beta)(\bar n)\Bigr]\eta
 -\Lambda_r\langle B^\top,\nabla^M\eta\rangle_\alpha.
\label{eq:Lambdadot}
\end{align}
The Riemannian hypersurface measure satisfies
\begin{equation}\label{eq:Riemann-area-variation}
 \delta dV_{\bar n}^\alpha=-H_\alpha\eta dV_{\bar n}^\alpha.
\end{equation}
Since $dV_n^F=\Lambda dV_{\bar n}^\alpha$,
\begin{align}
 \delta A_F
 &=\int_M(\dot\Lambda-\Lambda H_\alpha\eta)dV_{\bar n}^\alpha\notag\\
 &=\int_M
 \Bigl[\Lambda_q\bar n(q)+\Lambda_r(\nabla^\alpha_{\bar n}\beta)(\bar n)
 -\Lambda H_\alpha\Bigr]\eta\,dV_{\bar n}^\alpha\notag\\
 &\quad-
 \int_M\Lambda_r\langle B^\top,\nabla^M\eta\rangle_\alpha
 dV_{\bar n}^\alpha.
\label{eq:first-var-intermediate}
\end{align}
Put $Y=\Lambda_rB^\top$.  The tangential product rule
\begin{equation*}
 \diver_M(\eta Y)=\langle\nabla^M\eta,Y\rangle_\alpha
 +\eta\diver_MY
\end{equation*}
and the divergence theorem yield
\begin{equation*}
 -\int_M\langle Y,\nabla^M\eta\rangle_\alpha dV_{\bar n}^\alpha
 =\int_M\eta\diver_MY\,dV_{\bar n}^\alpha.
\end{equation*}
Thus
\begin{equation}\label{eq:first-var-Q}
 \delta A_F
 =\int_M(\mathcal Q_M-\Lambda H_\alpha)\eta
 dV_{\bar n}^\alpha.
\end{equation}
On the other hand, by \eqref{eq:first-variation-def} and
\eqref{eq:flux-comparison},
\begin{equation}\label{eq:first-var-HF}
 \delta A_F
 =-\int_M\rho H_F\eta\,dV_{\bar n}^\alpha.
\end{equation}
Comparison of \eqref{eq:first-var-Q} and \eqref{eq:first-var-HF}, together with
the arbitrariness of $\eta$, proves \eqref{eq:general-H}.
\end{proof}

The same identity admits a more economical form.  Extend $\bar n$ locally by
the $\alpha$-geodesic normal foliation, so that
$\nabla^\alpha_{\bar n}\bar n=0$, and use this extension when differentiating
$r=\beta(\bar n)$ and $\psi(q,r)$ in the normal direction.

\begin{proposition}[Reduced volume mean-curvature formula]\label{prop:reduced-H}
Define
\begin{equation}\label{eq:C-def}
 \boxed{\mathcal C_M=\diver_M(\psi_rB^\top)+n(\log\rho).}
\end{equation}
Then
\begin{equation}\label{eq:reduced-H}
 \boxed{H_F=\psi H_\alpha-\bar n(\psi)-\mathcal C_M.}
\end{equation}
\end{proposition}

\begin{proof}
For $dV_F=\rho dV_\alpha$,
\begin{equation*}
 \diver_{dV_F}Y=\diver_\alpha Y+Y(\log\rho).
\end{equation*}
Using \eqref{eq:normal-comparison},
\begin{equation*}
 n=\psi\bar n+\psi_rB^\top.
\end{equation*}
Since $\psi_rB^\top$ is tangent to the leaves of the chosen geodesic normal
extension,
\begin{equation*}
 \diver_\alpha(\psi_rB^\top)=\diver_M(\psi_rB^\top).
\end{equation*}
Also
\begin{equation*}
 \diver_\alpha(\psi\bar n)
 =\bar n(\psi)+\psi\diver_\alpha\bar n
 =\bar n(\psi)-\psi H_\alpha.
\end{equation*}
Therefore
\begin{align*}
 H_F
 &=-\diver_{dV_F}n\\
 &=\psi H_\alpha-\bar n(\psi)
 -\diver_M(\psi_rB^\top)-n(\log\rho),
\end{align*}
which is \eqref{eq:reduced-H}.
\end{proof}

The two terms in $\mathcal C_M$ have distinct geometric origins: the first
records the tangential component of the Finsler normal, while the second records
the change of the chosen ambient density along the Finsler normal.  Their sum is
the natural correction term in the weakest same-foliation transfer condition.

\section{All-normal induced-volume rigidity}\label{sec:rigidity}

The induced-volume identity has a rigid consequence when it is required in
every normal direction at a point.  This should be distinguished from the much
weaker requirement imposed only along one prescribed hypersurface family.

\begin{theorem}[All-normal induced-volume rigidity]\label{thm:rigidity}
Let $F_x$ be any smooth strongly convex Finsler norm on $T_xN$, and let the
ambient Finsler density be BH or HT.  Suppose that for every cooriented
hyperplane through $x$,
\begin{equation}\label{eq:rigidity-assumption}
 dV_n^F=L(x)dV_{\bar n}^\alpha
\end{equation}
with the same positive factor $L(x)$, independent of the hyperplane normal.
Then
\begin{equation}\label{eq:rigidity-conclusion}
 \boxed{F_x=L(x)^{1/(n-1)}\alpha_x.}
\end{equation}
In particular, if $L(x)=1$, then $F_x=\alpha_x$.  If the hypothesis holds
smoothly at every point, then $F$ is Riemannian and conformal to $\alpha$.
\end{theorem}

\begin{proof}
The argument does not require the $(\alpha,\beta)$ form.  Let $\bar\nu$ be an
$\alpha$-unit conormal to a hyperplane.  The corresponding $F$-unit conormal is
\begin{equation*}
 \nu=\frac{\bar\nu}{F_x^*(\bar\nu)}.
\end{equation*}
If the ambient density ratio at $x$ is $\rho(x)$, the same contraction argument
as in Lemma~\ref{lem:normal-volume} gives
\begin{equation}\label{eq:general-normal-volume}
 dV_n^F=\rho(x)F_x^*(\bar\nu)dV_{\bar n}^\alpha.
\end{equation}
Under \eqref{eq:rigidity-assumption},
\begin{equation*}
 F_x^*(\bar\nu)=\frac{L(x)}{\rho(x)}
\end{equation*}
is constant on the entire dual $\alpha$-unit sphere.  Positive
one-homogeneity therefore gives
\begin{equation*}
 F_x^*=c\alpha_x^*,
\end{equation*}
so $F_x=A\alpha_x$ with $A=c^{-1}$.  For BH and HT volume, scaling the norm by
$A$ multiplies the ambient density by $A^n$, while the unit-conormal
normalization contributes $A^{-1}$.  Hence the induced hypersurface factor is
$A^{n-1}$.  Thus $L=A^{n-1}$ and
$A=L^{1/(n-1)}$, proving \eqref{eq:rigidity-conclusion}.
\end{proof}

\begin{corollary}[Exact all-normal equality]\label{cor:rigidity-exact}
If
\begin{equation*}
 dV_n^F=dV_{\bar n}^\alpha
\end{equation*}
for every admissible hypersurface normal direction at every point, then
$F=\alpha$.  In particular, with the same orientation,
$H_F=H_\alpha$ pointwise and the corresponding volume-preserving variational
problems coincide.
\end{corollary}

The quantifier ``every normal direction'' is essential.  In Section
\ref{sec:example} we shall impose exact induced-volume equality only along one
radial foliation; the resulting metric remains genuinely non-Riemannian.

\section{Two isoparametric transfer problems}\label{sec:transfer}

\subsection{Same-hypersurface transfer}

Let $M\subset N$ be a connected cooriented hypersurface which is known to be
$\alpha$-isoparametric.  Fix the chosen $F$-normal side.  After shrinking if
necessary, let
\begin{equation}\label{eq:F-tube}
 U^+=\{\exp_x^F(\tau n_x):x\in M,\ 0\le\tau<\varepsilon\}
\end{equation}
be a regular one-sided $F$-normal tube and let $d_F$ be its forward normal
distance.  Put
\begin{equation}\label{eq:F-leaves}
 M_\tau^F=\{d_F=\tau\},\qquad M_0^F=M.
\end{equation}
We assume the regular leaves are connected.  For each leaf, let $\bar n_\tau$
be its same-side $\alpha$-unit normal and set
\begin{equation}\label{eq:levelI-data}
 q_\tau=\|\beta\|_\alpha^2,
 \qquad r_\tau=\beta(\bar n_\tau),
 \qquad B_\tau^\top=B-r_\tau\bar n_\tau.
\end{equation}
Every geometric quantity below is evaluated on the actual $F$-parallel leaf
$M_\tau^F$; in particular, $H_{\alpha,\tau}$ is the $\alpha$-mean curvature of
$M_\tau^F$, not of the $\alpha$-parallel leaf at the same numerical distance.

\begin{definition}[Same-hypersurface transfer]\label{def:levelI}
The hypersurface $M$ satisfies same-hypersurface transfer on the chosen side if the same
cooriented subset $M\subset N$ is locally $F$-isoparametric.  No equality
between the $\alpha$-normal tube and the $F$-normal tube is required.
\end{definition}

\begin{theorem}[Exact local same-hypersurface criterion]\label{thm:local-levelI}
The hypersurface $M$ satisfies same-hypersurface transfer on the chosen side if and only
if, for every regular leaf of its $F$-normal tube,
\begin{equation}\label{eq:levelI-criterion}
 \boxed{
 d_{M_\tau^F}
 \left[
 \frac{\Lambda(q_\tau,r_\tau)H_{\alpha,\tau}
 -\mathcal Q_{M_\tau^F}}
 {\rho(q_\tau)}
 \right]=0.
 }
\end{equation}
Equivalently, every regular $F$-parallel leaf has constant $F$-volume mean
curvature.
\end{theorem}

\begin{proof}
The forward distance coordinate satisfies
\begin{equation}\label{eq:dF-transnormal}
 F^*(dd_F)=1
\end{equation}
on the regular tube, so it is automatically $F$-transnormal.  If $d_F$ is
$F$-isoparametric, then $M=M_0^F$ is locally $F$-isoparametric.

Conversely, suppose that $M$ is locally $F$-isoparametric.  Let $u$ be a local
$F$-isoparametric defining function with $M=\{u=c\}$ and
\begin{equation*}
 F^*(du)=a_F(u)>0.
\end{equation*}
Define
\begin{equation}\label{eq:distance-reparam}
 \Theta(t)=\int_c^t\frac{d\sigma}{a_F(\sigma)},
 \qquad d=\Theta(u).
\end{equation}
Positive homogeneity gives
\begin{equation*}
 F^*(dd)=\Theta'(u)F^*(du)=1.
\end{equation*}
The normal-geodesic property of regular Finsler transnormal functions implies,
after restricting to the chosen side, that $d$ is exactly the local forward
normal-distance coordinate $d_F$.  Since increasing regular
reparametrizations preserve isoparametricity, $d_F$ is $F$-isoparametric.
Thus same-hypersurface transfer is equivalent to the $F$-isoparametricity of $d_F$.

By Lemma~\ref{lem:transnormal-CMC}, because \eqref{eq:dF-transnormal} already
provides transnormality, $d_F$ is $F$-isoparametric if and only if each connected
regular level $M_\tau^F$ has constant $F$-volume mean curvature.  Proposition
\ref{prop:general-H} gives
\begin{equation}\label{eq:levelI-H}
 H_{F,\tau}
 =\frac{\Lambda(q_\tau,r_\tau)H_{\alpha,\tau}
 -\mathcal Q_{M_\tau^F}}{\rho(q_\tau)}.
\end{equation}
Since $M_\tau^F$ is connected, constancy is equivalent to the vanishing of the
tangential differential, which is exactly \eqref{eq:levelI-criterion}.
\end{proof}

\subsubsection*{Randers specialization}

Let $F$ now be a Randers metric with navigation datum $(h,W)$ and
$|W|_h<1$, equipped with BH volume.  In the generalized
$(\alpha,\beta)$ notation we set
\begin{equation}\label{eq:randers-dictionary}
 \alpha=h,
 \qquad
 \beta=W_h^\flat,
 \qquad
 B=W.
\end{equation}
The navigation dual norm is
\begin{equation}\label{eq:randers-dual}
 F^*(\xi)=h^*(\xi)+\xi(W).
\end{equation}
Therefore, with $r=\xi(W)/h^*(\xi)$,
\begin{equation}\label{eq:randers-psi}
 \boxed{\psi(q,r)=1+r,\qquad \psi_q=0,\qquad \psi_r=1.}
\end{equation}
The $F$-unit ball is the translation by $W$ of the $h$-unit ball, so the BH
density agrees with the Riemannian density:
\begin{equation}\label{eq:randers-rho}
 \boxed{dV_F=dV_h,\qquad \rho=1.}
\end{equation}
Hence
\begin{equation}\label{eq:randers-Lambda}
 \Lambda=1+r,
 \qquad \Lambda_q=0,
 \qquad \Lambda_r=1,
\end{equation}
and the normal formula \eqref{eq:normal-comparison} becomes
\begin{equation}\label{eq:randers-normal}
 \boxed{n=(1+r)\bar n+(W-r\bar n)=\bar n+W.}
\end{equation}

Assume that $F$ has isotropic $S$-curvature
$S=(n+1)k(x)F$.  In the convention of \cite{HeDongYin2020}, this is equivalent
to
\begin{equation}\label{eq:randers-homothetic}
 \frac12(w_{i|j}+w_{j|i})=-2k(x)h_{ij}.
\end{equation}
Because $\beta=W_h^\flat$,
\begin{equation}\label{eq:randers-normal-beta}
 (\nabla^h_{\bar n}\beta)(\bar n)
 =h(\nabla^h_{\bar n}W,\bar n)=-2k(x).
\end{equation}
Choose an $h$-orthonormal tangent frame $e_1,\ldots,e_{n-1}$.  Since
$W^\top=W-r\bar n$,
\begin{align}
 \diver_MW^\top
 &=\sum_{a=1}^{n-1}h(\nabla^h_{e_a}(W-r\bar n),e_a)\notag\\
 &=\sum_{a=1}^{n-1}h(\nabla^h_{e_a}W,e_a)
 -r\sum_{a=1}^{n-1}h(\nabla^h_{e_a}\bar n,e_a)\notag\\
 &=-2(n-1)k(x)+rH_h,
\label{eq:randers-divW}
\end{align}
where $H_h=-\diver_h\bar n$.  Thus
\begin{equation}\label{eq:randers-Q}
 \boxed{\mathcal Q_M=-2nk(x)+rH_h.}
\end{equation}
Substitution in \eqref{eq:general-H} gives the same-hypersurface identity
\begin{equation}\label{eq:randers-H-volume}
 \boxed{H_F^{\mathrm{vol}}=H_h+2nk(x).}
\end{equation}
For comparison with the anisotropic shape operator of Section 2, the Randers
shape-operator identity of \cite{HeDongYin2020} gives
$\widehat H_F=H_h+(n-1)k(x)$, while
$S(n)=(n+1)k(x)$; hence
$H_F^{\mathrm{vol}}=\widehat H_F+S(n)$ in these conventions.

For a Randers space form, $k(x)=k_0$ is constant.  The same-hypersurface criterion then
reduces to
\begin{equation}\label{eq:randers-levelI-criterion}
 d_{M_\tau^F}H_h(M_\tau^F)=0.
\end{equation}
It remains to compare the actual $F$-normal tube with the $h$-normal tube.
The homothetic equation becomes
\begin{equation}\label{eq:LieW-constant}
 \mathcal L_Wh=-4k_0h.
\end{equation}
Let $\varphi_\tau$ be the local flow of $W$.  Then
\begin{equation}\label{eq:flow-homothety-main}
 \boxed{\varphi_\tau^*h=e^{-4k_0\tau}h,}
\end{equation}
so $\varphi_\tau$ scales $h$-lengths by
$c_\tau=e^{-2k_0\tau}$.

Let
\begin{equation*}
 \rho_x(s)=\exp_x^h(s\bar n_x)
\end{equation*}
be the unit-speed $h$-normal geodesic and put
\begin{equation}\label{eq:s-of-tau}
 s(\tau)=\int_0^\tau e^{2k_0u}\,du
 =\begin{cases}
 \dfrac{e^{2k_0\tau}-1}{2k_0},&k_0\ne0,\\[2mm]
 \tau,&k_0=0.
 \end{cases}
\end{equation}
Then $s'(\tau)=e^{2k_0\tau}$.  The curve
\begin{equation}\label{eq:randers-normal-geodesic}
 \gamma_x(\tau)=\varphi_\tau(\rho_x(s(\tau)))
\end{equation}
satisfies
\begin{equation*}
 \dot\gamma_x=W+s'(\tau)d\varphi_\tau(\rho_x'),
 \qquad
 |\dot\gamma_x-W|_h=s'(\tau)e^{-2k_0\tau}=1.
\end{equation*}
Thus the homothetic navigation geodesic theorem identifies it with the
$F$-unit geodesic issued from $x$ in the normal direction
$n_x=\bar n_x+W_x$.  Consequently
\begin{equation}\label{eq:randers-tube-flow}
 \boxed{M_\tau^F=\varphi_\tau(M_{s(\tau)}^h).}
\end{equation}

We record explicitly how the mean curvature transforms under the homothety.
If a diffeomorphism $\varphi$ satisfies $\varphi^*h=c^2h$ with constant $c>0$,
then $\varphi:(N,c^2h)\to(N,h)$ is an isometry.  Constant scaling does not
change the Levi--Civita connection.  If $\bar n$ is the $h$-unit normal to
$\Sigma$, then the unit normal to $\varphi(\Sigma)$ is
$c^{-1}d\varphi(\bar n)$, and
\begin{equation*}
 A_{\varphi(\Sigma)}(d\varphi X)
 =c^{-1}d\varphi(A_\Sigma X).
\end{equation*}
Hence every principal curvature, and therefore the unnormalized mean curvature,
is multiplied by $c^{-1}$.  Applying this to \eqref{eq:flow-homothety-main},
\begin{equation}\label{eq:randers-H-flow}
 H_h(M_\tau^F)
 =e^{2k_0\tau}H_h(M_{s(\tau)}^h)\circ\varphi_\tau^{-1}.
\end{equation}
If $M$ is $h$-isoparametric, the right-hand side is constant on each
$M_\tau^F$, so \eqref{eq:randers-levelI-criterion} holds.  Conversely, if the
$F$-tube is $h$-CMC, the inverse homothety and the local diffeomorphism
$\tau\mapsto s(\tau)$ show that all sufficiently small $h$-parallel leaves are
$h$-CMC.  The transnormal--CMC criterion for the $h$-distance then gives
\begin{equation}\label{eq:randers-levelI-equivalence}
 \boxed{
 M\text{ is locally }F\text{-isoparametric}
 \Longleftrightarrow
 M\text{ is locally }h\text{-isoparametric}.
 }
\end{equation}
This is the hypersurface-level mechanism behind Theorem 1.1 of
\cite{HeDongYin2020}.  The equality of the numbers and multiplicities of
principal curvatures in that theorem is stronger than the scalar mean-curvature
argument above and follows from the Randers shape-operator relation
$A^F=A^h+k_0I$.

The equivalence \eqref{eq:randers-levelI-equivalence}, together with the
classical Riemannian classification in real space forms, gives the
same-hypersurface mechanism used in the Randers classification theorem.

\subsubsection*{Global same-hypersurface transfer}

The local criterion has a direct global form once the prescribed hypersurface
is assumed to belong to a globally defined $F$-transnormal system.

\begin{theorem}[Global same-hypersurface transfer]\label{thm:global-levelI}
Let $M\subset N$ be a connected cooriented $\alpha$-isoparametric
hypersurface.  Let $g:N\to\mathbb R$ be a non-constant $C^1$ function, smooth
on $N_g=\{dg\ne0\}$, whose connected regular levels are $F$-transnormal, and
suppose that $M=g^{-1}(t_0)$ for a regular value $t_0$.  Denote the regular
levels by $M_t^F=g^{-1}(t)$.  Then $g$ is $F$-isoparametric if and only if, for
every regular value $t$,
\begin{equation}\label{eq:global-levelI-criterion}
 \boxed{
 d_{M_t^F}
 \left[
 \frac{\Lambda(q_t,r_t)H_{\alpha,t}-\mathcal Q_{M_t^F}}
 {\rho(q_t)}
 \right]=0.}
\end{equation}
Here $H_{\alpha,t}$ is the $\alpha$-mean curvature of the actual leaf
$M_t^F$.
\end{theorem}

\begin{proof}
The function $g$ is already $F$-transnormal on $N_g$.  Hence, by
Lemma~\ref{lem:transnormal-CMC}, $g$ is $F$-isoparametric if and only if every
connected regular level $M_t^F$ has constant $F$-volume mean curvature.  On
$M_t^F$, Proposition~\ref{prop:general-H} gives
\[
 H_{F,t}=
 \frac{\Lambda(q_t,r_t)H_{\alpha,t}-\mathcal Q_{M_t^F}}
 {\rho(q_t)}.
\]
Since the leaf is connected, its $F$-volume mean curvature is constant if and
only if the tangential differential of the right-hand side vanishes.  This is
exactly \eqref{eq:global-levelI-criterion}.
\end{proof}

Notice that the global function $g$ is not required to be an
$\alpha$-isoparametric defining function and its regular levels need not agree
with the $\alpha$-parallel family of $M$.  The global statement therefore
retains the same-hypersurface character of the local problem.

\subsection{Same-foliation transfer}

Let $U$ be a saturated regular neighborhood and let
\begin{equation}\label{eq:levelII-f}
 f:U\to I
\end{equation}
have connected regular leaves
\begin{equation*}
 M_t=f^{-1}(t).
\end{equation*}
When $f$ is $\alpha$-transnormal we write
\begin{equation}\label{eq:alpha-transnormal-levelII}
 \alpha^*(df)=a(f)>0,
 \qquad
 \bar n=\frac{\nabla^\alpha f}{a(f)},
 \qquad
 r=\beta(\bar n).
\end{equation}

\begin{definition}[Same-foliation transfer]\label{def:levelII}
A regular cooriented foliation satisfies same-foliation transfer if the same
foliation is isoparametric for $\alpha$ and $F$.  Since increasing regular
reparametrizations do not change isoparametricity, this is a property of the
foliation and not of a particular numerical defining function.
\end{definition}

The exact same-foliation theorem below combines the local criterion and the two-way
transfer statement into one result.

\begin{theorem}[Exact and two-way same-foliation criterion]\label{thm:levelII-unified}
Let $f$ have connected regular levels and define
\begin{equation}\label{eq:C-levelII}
 \mathcal C_{M_t}=\diver_{M_t}(\psi_rB^\top)+n(\log\rho).
\end{equation}
Then the following statements hold.

\begin{enumerate}[label=\textup{(\alph*)},leftmargin=2.5em]
\item Suppose $f$ is $\alpha$-isoparametric,
\begin{equation}\label{eq:alpha-iso-levelII}
 \alpha^*(df)=a(f)>0,
 \qquad
 \Delta_\alpha f=b_\alpha(f).
\end{equation}
Then the same foliation is $F$-isoparametric if and only if, on every regular
leaf,
\begin{equation}\label{eq:levelII-exact}
 \boxed{d_{M_t}\psi(q,r)=0,
 \qquad
 d_{M_t}\mathcal C_{M_t}=0.}
\end{equation}

\item More generally, suppose there exist one-variable functions $\psi>0$ and
$C$ such that
\begin{equation}\label{eq:reduced-compatible}
 \psi(q,r)=\psi(f),
 \qquad
 \mathcal C_{M_t}=C(f)
\end{equation}
on the regular region.  Then
\begin{equation}\label{eq:two-way-levelII}
 \boxed{f\text{ is }F\text{-isoparametric}
 \Longleftrightarrow
 f\text{ is }\alpha\text{-isoparametric}.}
\end{equation}
Whenever $\alpha^*(df)=a(f)$,
\begin{equation}\label{eq:aF-reduced}
 \boxed{F^*(df)=a_F(f)=a(f)\psi(f),}
\end{equation}
\begin{equation}\label{eq:HF-reduced-levelII}
 \boxed{H_F=\psi(f)H_\alpha-a(f)\psi'(f)-C(f).}
\end{equation}
If $f$ is $\alpha$-isoparametric, then
\begin{equation}\label{eq:DeltaF-levelII}
 \boxed{
 \Delta_Ff
 =\psi(f)^2b_\alpha(f)
 +2a(f)^2\psi(f)\psi'(f)
 +a(f)\psi(f)C(f).
 }
\end{equation}
\end{enumerate}
\end{theorem}

\begin{proof}
Assume first that $f$ is $\alpha$-isoparametric and that the same foliation is
$F$-isoparametric.  Let $\widetilde f$ be a local $F$-isoparametric defining
function whose connected level sets are exactly the $M_t$ with the same
coorientation.  Since $f$ and $\widetilde f$ are submersions with identical
connected level sets on a saturated neighborhood, there is an increasing
smooth function $\chi$ such that
\begin{equation}\label{eq:levelII-reparam}
 \widetilde f=\chi(f),
 \qquad \chi'>0.
\end{equation}
If $F^*(d\widetilde f)=\widetilde a_F(\widetilde f)$, then homogeneity gives
\begin{equation*}
 \chi'(f)F^*(df)=\widetilde a_F(\chi(f)),
\end{equation*}
so $F^*(df)$ is a function of $f$.  But
\begin{equation}\label{eq:Fstar-df-general}
 F^*(df)=\alpha^*(df)\psi(q,r)=a(f)\psi(q,r).
\end{equation}
Since $a(f)$ is already a leaf function, $\psi(q,r)$ is constant on every
connected leaf.  Thus the first condition in \eqref{eq:levelII-exact} holds and
we may write $\psi(q,r)=\psi(f)$.

The $\alpha$-isoparametricity of $f$ implies that $H_\alpha$ is constant on
each regular level.  Since $\bar n(f)=a(f)$,
\begin{equation}\label{eq:normal-psi-levelII}
 \bar n(\psi)=a(f)\psi'(f).
\end{equation}
Proposition~\ref{prop:reduced-H} therefore gives
\begin{equation}\label{eq:HF-levelII-proof}
 H_F=\psi(f)H_\alpha-a(f)\psi'(f)-\mathcal C_{M_t}.
\end{equation}
The common foliation is $F$-isoparametric, so $H_F$ is also constant on each
leaf.  Every other term in \eqref{eq:HF-levelII-proof} is a leaf function;
therefore $\mathcal C_{M_t}$ is constant on each leaf.  This proves the
necessity of \eqref{eq:levelII-exact}.

Conversely, assume \eqref{eq:levelII-exact}.  Connectedness yields
$\psi(q,r)=\psi(f)>0$, and \eqref{eq:Fstar-df-general} shows that $f$ itself is
$F$-transnormal.  Formula \eqref{eq:HF-levelII-proof}, together with the
$\alpha$-isoparametricity of $f$ and the second condition in
\eqref{eq:levelII-exact}, shows that every regular leaf is $F$-volume CMC.
Lemma~\ref{lem:transnormal-CMC} implies that $f$ is $F$-isoparametric.  This
proves part (a).

For part (b), \eqref{eq:reduced-compatible} gives immediately
\begin{equation*}
 F^*(df)=\psi(f)\alpha^*(df).
\end{equation*}
Because $\psi(f)>0$, $F$-transnormality and $\alpha$-transnormality are
equivalent.  Under either transnormal condition write
$\alpha^*(df)=a(f)$.  Then \eqref{eq:normal-psi-levelII} and
Proposition~\ref{prop:reduced-H} give \eqref{eq:HF-reduced-levelII}.  Since all
coefficients except the mean curvature are functions of $f$,
\begin{equation*}
 H_F|_{M_t}\text{ is constant}
 \Longleftrightarrow
 H_\alpha|_{M_t}\text{ is constant}.
\end{equation*}
Together with the transnormal--CMC criterion this proves
\eqref{eq:two-way-levelII}.

Finally, if $f$ is $\alpha$-isoparametric, then
\begin{equation*}
 H_\alpha=a'(f)-\frac{b_\alpha(f)}{a(f)}.
\end{equation*}
Using $a_F=a\psi$ and
$\Delta_Ff=a_Fa_F'-a_FH_F$, we compute
\begin{align*}
 \Delta_Ff
 &=a\psi(a'\psi+a\psi')
 -a\psi(\psi H_\alpha-a\psi'-C)\\
 &=\psi^2a(a'-H_\alpha)+2a^2\psi\psi'+a\psi C\\
 &=\psi^2b_\alpha+2a^2\psi\psi'+a\psi C,
\end{align*}
which proves \eqref{eq:DeltaF-levelII}.
\end{proof}

\subsubsection*{Randers specialization}

Keep the Randers navigation notation \eqref{eq:randers-dictionary} and assume
that $W$ is homothetic with constant parameter $k_0$.  Let $f$ be an
$h$-isoparametric function with connected regular leaves and write
\begin{equation}\label{eq:randers-h-isof}
 h^*(df)=a(f)>0,
 \qquad
 \Delta_hf=b(f),
 \qquad
 \bar n=\frac{\nabla^hf}{a(f)}.
\end{equation}
Then
\begin{align}
 df(W)
 &=h(\nabla^hf,W)\notag\\
 &=a(f)h(\bar n,W)\notag\\
 &=a(f)r,
\end{align}
so
\begin{equation}\label{eq:randers-r-dfW}
 \boxed{r=\frac{df(W)}{a(f)}.}
\end{equation}
Since $\psi=1+r$,
\begin{equation}\label{eq:randers-psi-dfW}
 \psi(q,r)=1+\frac{df(W)}{a(f)}.
\end{equation}
On a fixed connected leaf $a(f)=a(t)$ is constant, and hence
\begin{equation}\label{eq:randers-first-obstruction}
 d_{M_t}\psi=0
 \Longleftrightarrow
 d_{M_t}(df(W))=0
 \Longleftrightarrow
 df(W)=\varphi(f)
\end{equation}
for a smooth one-variable function $\varphi$.  Moreover,
\begin{equation}\label{eq:randers-aF-levelII}
 \boxed{F^*(df)=h^*(df)+df(W)=a(f)+\varphi(f).}
\end{equation}

The second obstruction is automatic.  Since $\rho=1$, $\psi_r=1$, and $B=W$,
\begin{equation}\label{eq:randers-C}
 \mathcal C_{M_t}=\diver_{M_t}W^\top.
\end{equation}
Formula \eqref{eq:randers-divW} gives
\begin{equation}\label{eq:randers-C-expanded}
 \mathcal C_{M_t}=-2(n-1)k_0+rH_h.
\end{equation}
The $h$-isoparametric equation implies
\begin{align}
 b(f)=\Delta_hf
 &=\diver_h(a(f)\bar n)\notag\\
 &=a(f)a'(f)-a(f)H_h,
\end{align}
therefore
\begin{equation}\label{eq:Hh-leaf-function}
 H_h=a'(f)-\frac{b(f)}{a(f)}.
\end{equation}
Under $df(W)=\varphi(f)$,
\begin{equation}\label{eq:randers-C-one-variable}
 \mathcal C_{M_t}
 =-2(n-1)k_0
 +\frac{\varphi(f)}{a(f)}
 \left(a'(f)-\frac{b(f)}{a(f)}\right),
\end{equation}
which is a function of $f$ alone.  Hence the second condition in
\eqref{eq:levelII-exact} follows automatically from the first.

There is also a useful normal derivative identity.  Since an
$h$-transnormal function has geodesic unit normal,
$\nabla^h_{\bar n}\bar n=0$.  Thus
\begin{equation*}
 \bar n(r)
 =h(\nabla^h_{\bar n}W,\bar n)=-2k_0.
\end{equation*}
If $r=\varphi(f)/a(f)$ and $\bar n(f)=a(f)$, then
\begin{equation}\label{eq:randers-rprime}
 \boxed{a(f)r'(f)=-2k_0,}
\end{equation}
or equivalently
\begin{equation}\label{eq:randers-phi-ode}
 \varphi'(f)-\frac{a'(f)}{a(f)}\varphi(f)=-2k_0.
\end{equation}
Substituting into the reduced mean-curvature formula gives again
\begin{align*}
 H_F^{\mathrm{vol}}
 &=(1+r)H_h-a r'-[-2(n-1)k_0+rH_h]\\
 &=H_h+2nk_0,
\end{align*}
in agreement with \eqref{eq:randers-H-volume}.

We have therefore obtained the following exact specialization.

\begin{corollary}[Randers same-foliation criterion]\label{cor:randers-levelII}
Let $F$ be a Randers metric with navigation datum $(h,W)$ and BH volume, and
suppose $W$ is homothetic with constant parameter $k_0$ as in
\eqref{eq:randers-homothetic}.  Let $f$ be an $h$-isoparametric function with
connected regular leaves.  Then the same foliation is $F$-isoparametric if and
only if
\begin{equation}\label{eq:randers-thm39}
 \boxed{df(W)=\varphi(f)}
\end{equation}
for a smooth one-variable function $\varphi$.  In this case
\begin{equation*}
 F^*(df)=a(f)+\varphi(f),
 \qquad
 H_F^{\mathrm{vol}}=H_h+2nk_0.
\end{equation*}
\end{corollary}

This is precisely Theorem 3.9 of \cite{HeDongYin2020}.  In the conventions of
this paper the corresponding Laplacian can also be written explicitly as
\begin{equation}\label{eq:randers-DeltaF}
 \Delta_Ff
 =(a+\varphi)
 \left(\frac{b}{a}+\varphi'-2nk_0\right).
\end{equation}

\subsubsection*{Global same-foliation transfer}

When the defining function itself is global, no additional condition at the
critical set is needed beyond the standard $C^1$ regularity in the definition
of an isoparametric function.

\begin{theorem}[Global same-foliation transfer]\label{thm:global-levelII}
Let $f:N\to\mathbb R$ be a non-constant $C^1$ function, smooth on
$N_f=\{df\ne0\}$, whose connected regular levels form an
$\alpha$-isoparametric foliation.  On $N_f$ write
\begin{equation}\label{eq:global-alpha-iso}
 \alpha^*(df)=a_\alpha(f)>0,
 \qquad
 \Delta_\alpha f=b_\alpha(f).
\end{equation}
Then the same global function $f$ is $F$-isoparametric if and only if, on every
regular leaf $M_t=f^{-1}(t)$,
\begin{equation}\label{eq:global-levelII-exact}
 \boxed{d_{M_t}\psi(q,r)=0,
 \qquad
 d_{M_t}\mathcal C_{M_t}=0.}
\end{equation}
When these conditions hold, there are one-variable functions $\psi>0$ and
$C$ such that
\begin{equation}\label{eq:global-reduced-functions}
 \psi(q,r)=\psi(f),
 \qquad
 \mathcal C_{M_t}=C(f)
\end{equation}
on $N_f$, and
\begin{equation}\label{eq:global-aF}
 F^*(df)=a_F(f)=a_\alpha(f)\psi(f).
\end{equation}
Moreover,
\begin{equation}\label{eq:global-HF}
 H_F=\psi(f)H_\alpha-a_\alpha(f)\psi'(f)-C(f),
\end{equation}
and
\begin{equation}\label{eq:global-DeltaF}
 \Delta_Ff
 =\psi(f)^2b_\alpha(f)
 +2a_\alpha(f)^2\psi(f)\psi'(f)
 +a_\alpha(f)\psi(f)C(f).
\end{equation}
\end{theorem}

\begin{proof}
On each saturated regular region, Theorem~\ref{thm:levelII-unified} applies.
If $f$ is $F$-isoparametric, the necessity of
\eqref{eq:global-levelII-exact} follows from that theorem.  Conversely, if
\eqref{eq:global-levelII-exact} holds on every connected regular leaf, then
connectedness gives \eqref{eq:global-reduced-functions}.  Hence
\[
 F^*(df)=\alpha^*(df)\psi(f)=a_\alpha(f)\psi(f),
\]
so $f$ is $F$-transnormal on $N_f$.  The reduced mean-curvature formula gives
\eqref{eq:global-HF}; its right-hand side is constant on each regular leaf.
Thus every regular leaf is $F$-volume CMC and Lemma~\ref{lem:transnormal-CMC}
shows that $f$ is $F$-isoparametric.  Since the same function $f$ is already
$C^1$ on all of $N$ and smooth on $N_f$, this is precisely a global
isoparametric function in the sense of Section~\ref{subsec:isoparametric-prelim}.
Formula \eqref{eq:global-DeltaF} is the global form of
\eqref{eq:DeltaF-levelII}.
\end{proof}

\subsection{Useful sufficient structures}\label{subsec:sufficient}

The exact same-foliation theorem only requires the two combinations $\psi$ and
$\mathcal C_M$ to be leaf functions.  In applications it is often convenient
to impose stronger conditions that are easier to verify directly from the
induced-volume formula.

\begin{definition}[Induced-volume compatibility]\label{def:induced-compatible}
A regular cooriented foliation defined by $f$ is induced-volume compatible if
there exist one-variable functions $\varrho>0$, $L>0$, and $J$ such that
\begin{equation}\label{eq:induced-compatible}
 \rho(q)=\varrho(f),
 \qquad
 \Lambda(q,r)=L(f),
 \qquad
 \diver_{M_t}(\Lambda_rB^\top)=J(f).
\end{equation}
\end{definition}

\begin{definition}[Strong compatibility]\label{def:strong-compatible}
The foliation is strongly compatible if there exist one-variable functions
$Q$, $R$, and $K$ such that
\begin{equation}\label{eq:strong-compatible}
 q=Q(f),
 \qquad
 r=R(f),
 \qquad
 \psi_r(Q(f),R(f))\diver_{M_t}B^\top=K(f).
\end{equation}
\end{definition}

\begin{corollary}[Practical sufficient conditions for same-foliation transfer]\label{cor:practical}
The following statements hold.
\begin{enumerate}[label=\textup{(\roman*)},leftmargin=2.5em]
\item Every induced-volume compatible regular defining function satisfies
\begin{equation}\label{eq:induced-two-way}
 f\text{ is }F\text{-isoparametric}
 \Longleftrightarrow
 f\text{ is }\alpha\text{-isoparametric}.
\end{equation}
If $\alpha^*(df)=a(f)$, then
\begin{equation}\label{eq:induced-psi}
 \psi(f)=\frac{L(f)}{\varrho(f)},
\end{equation}
\begin{equation}\label{eq:induced-H}
 \boxed{\varrho H_F=LH_\alpha-aL'-J,}
\end{equation}
and, when $f$ is $\alpha$-isoparametric,
\begin{equation}\label{eq:induced-Delta}
 \boxed{
 \Delta_Ff
 =\frac{L^2b_\alpha+2a^2LL'+aLJ}{\varrho^2}
 -\frac{a^2L^2\varrho'}{\varrho^3}.}
\end{equation}

\item If in addition the induced hypersurface volumes agree exactly along the
foliation,
\begin{equation}\label{eq:exact-leaf-volume}
 dV_n^F=dV_{\bar n}^\alpha,
\end{equation}
then $L\equiv1$ and
\begin{equation}\label{eq:exact-specialization}
 F^*(df)=\frac{a}{\varrho},
 \qquad
 H_F=\frac{H_\alpha-J}{\varrho},
 \qquad
 \Delta_Ff=\frac{b_\alpha+aJ}{\varrho^2}
 -\frac{a^2\varrho'}{\varrho^3}.
\end{equation}
If every regular leaf is compact without boundary, then $J\equiv0$.

\item Every strongly compatible foliation is induced-volume compatible.  In
fact,
\begin{equation}\label{eq:strong-induced-data}
 \varrho(f)=\rho(Q(f)),
 \qquad
 L(f)=\varrho(f)\psi(Q(f),R(f)),
 \qquad
 J(f)=\varrho(f)K(f).
\end{equation}
Hence strong compatibility implies same-foliation transfer whenever either side is
isoparametric.
\end{enumerate}
\end{corollary}

\begin{proof}
Under \eqref{eq:induced-compatible}, the first two equalities give
$\psi=L/\varrho$.  Since $\Lambda_r=\rho\psi_r$,
\begin{equation*}
 \diver_{M_t}(\psi_rB^\top)=\frac{J}{\varrho}.
\end{equation*}
Moreover the $F$-unit normal satisfies
\begin{equation*}
 n(f)=F^*(df)=a\psi,
\end{equation*}
so
\begin{equation*}
 n(\log\rho)=a\psi\frac{\varrho'}{\varrho}.
\end{equation*}
Therefore
\begin{equation}\label{eq:C-induced-proof}
 \mathcal C_{M_t}=\frac{J}{\varrho}
 +a\psi\frac{\varrho'}{\varrho},
\end{equation}
which is a function of $f$.  Theorem~\ref{thm:levelII-unified} gives the two-way
transfer.  Substituting $\psi=L/\varrho$ into
\eqref{eq:HF-reduced-levelII} and \eqref{eq:DeltaF-levelII} yields
\eqref{eq:induced-H} and \eqref{eq:induced-Delta}.

Under exact equality \eqref{eq:exact-leaf-volume}, the factor
$dV_n^F=L(f)dV_{\bar n}^\alpha$ gives $L=1$, which immediately produces
\eqref{eq:exact-specialization}.  If each leaf is compact without boundary,
$J(t)$ is constant on the leaf and
\begin{equation*}
 J(t)\Vol_\alpha(M_t)
 =\int_{M_t}\diver_{M_t}(\Lambda_rB^\top)dV_{\bar n}^\alpha=0,
\end{equation*}
so $J=0$.

Finally, under strong compatibility, $\rho=\rho(Q(f))$ and
$\Lambda=\rho(Q(f))\psi(Q(f),R(f))$ are leaf functions.  Because $\Lambda_r=\rho\psi_r$
and $\rho$ is constant on each leaf,
\begin{equation*}
 \diver_{M_t}(\Lambda_rB^\top)
 =\rho(Q(f))\psi_r(Q(f),R(f))\diver_{M_t}B^\top
 =\varrho(f)K(f),
\end{equation*}
which proves \eqref{eq:strong-induced-data}.
\end{proof}

\section{A complete non-Randers global example}\label{sec:example}

Throughout, $\alpha=|y|$ is the Euclidean metric on $\R^n$, $n\ge2$, and the ambient Finsler measure is the Busemann--Hausdorff volume.

\subsection{A bounded radial one-form}

Fix $c>0$ and define
\begin{equation}\label{eq:h-function}
 h(z)=
 \begin{cases}
 1,&z\le0,\\[2mm]
 \exp\!\bigl(-ze^{-1/z}\bigr),&z>0.
 \end{cases}
\end{equation}
All derivatives of $ze^{-1/z}$ tend to zero as $z\downarrow0$, so $h$ is smooth.  For $t\ge0$ put
\begin{equation}\label{eq:qstar}
 q_*(t)=c^2\int_0^t h(16c^2\tau-1)\,\dd\tau,
 \qquad
 k(t)=\sqrt{\frac{q_*(t)}{t}}\quad(t>0),
 \qquad
 k(0)=c.
\end{equation}
Then $q_*'(t)>0$ and
\begin{equation}\label{eq:qstar-euclidean-region}
 q_*(t)=c^2t
 \qquad
 \text{for }0\le t\le\frac1{16c^2}.
\end{equation}
Moreover, for $z\ge1$ one has $e^{-1/z}\ge e^{-1}$ and hence
\[
 h(z)=\exp\!\bigl(-ze^{-1/z}\bigr)\le e^{-z/e}.
\]
Therefore
\[
 \int_0^\infty h(z)\,\dd z
 \le 1+\int_1^\infty e^{-z/e}\,\dd z
 =1+e^{1-1/e}<3,
\]
and consequently
\begin{equation}\label{eq:qstar-bound}
 q_*(t)<q_\infty
 :=\frac1{16}\left(1+\int_0^\infty h(z)\,\dd z\right)
 \le\frac{2+e^{1-1/e}}{16}<\frac14.
\end{equation}

Set
\begin{equation}\label{eq:beta-global}
 f(x)=|x|^2,
 \qquad
 \beta(x,y)=k(f)\langle x,y\rangle.
\end{equation}
Then
\begin{equation}\label{eq:q-global}
 q=\|\beta\|_\alpha^2=k(f)^2|x|^2=q_*(f)<\frac14.
\end{equation}
For every outward centered sphere $M_t=\{|x|^2=t\}$, $t>0$,
\begin{equation}\label{eq:radial-data}
 \nbar=\frac{x}{|x|},
 \qquad
 r=\beta(\nbar)=\sqrt q,
 \qquad
 B^\top=0.
\end{equation}
Thus $q$ and $r$ are functions of $f$ and the tangential $B^\top$-term vanishes identically on the centered-sphere foliation.

\subsection{A smooth non-Randers profile}

Set
\begin{equation}\label{eq:q0q1}
 q_0=\frac1{256},\qquad q_1=\frac1{64},
\end{equation}
and choose a smooth nondecreasing cutoff $\chi$ such that
\begin{equation}\label{eq:chi}
 0\le\chi\le1,
 \qquad
 \chi(q)=0\ (q\le q_0),
 \qquad
 \chi(q)=1\ (q\ge q_1).
\end{equation}
For instance, if $\eta(t)=0$ for $t\le0$ and $\eta(t)=e^{-1/t}$ for $t>0$, one may take
\[
 \chi(q)=\frac{\eta(q-q_0)}{\eta(q-q_0)+\eta(q_1-q)}.
\]
Define
\begin{equation}\label{eq:P0}
 P_0(q,s)=1+2\chi(q)s+s^2.
\end{equation}
Since $|s|\le\sqrt q<1/2$,
\begin{align}
 P_0&\ge(1-|s|)^2>0,\label{eq:P0-pos}\\
 P_0-s(P_0)_s&=1-s^2>0,\label{eq:P0-P}\\
 P_0-s(P_0)_s+(q-s^2)(P_0)_{ss}
 &=1+2q-3s^2\ge1-q>0.\label{eq:P0-D}
\end{align}
Hence $F_0=\alpha P_0(q,s)$ is a regular generalized $(\alpha,\beta)$-metric on the entire realized parameter range.

Let its BH density ratio be
\begin{equation}\label{eq:rho-seed}
 \rho_s(q)=
 \frac{I_n}{\displaystyle\int_0^\pi
 [1+2\chi(q)\sqrt q\cos\theta+q\cos^2\theta]^{-n}
 \sin^{n-2}\theta\,\dd\theta}.
\end{equation}
Near $q=0$, $\chi=0$, so the integrand is $(1+q\cos^2\theta)^{-n}$; hence $\rho_s$ is smooth at $0$ and $\rho_s(0)=1$.

Along the outward radial direction $s=\sqrt q$, define
\begin{equation}\label{eq:Tq}
 T(q)=1+2\chi(q)\sqrt q+q.
\end{equation}
Although $\sqrt q$ occurs in this expression, $\chi$ vanishes on a neighborhood of zero, so $T$ is smooth at $q=0$.  Put
\begin{equation}\label{eq:Astar}
 A_*(q)=\left(\frac{T(q)}{\rho_s(q)}\right)^{1/(n-1)},
 \qquad
 \Phi_*(q,s)=A_*(q)P_0(q,s).
\end{equation}
Then $A_*$ and $\Phi_*$ are smooth at $q=0$.

\subsection{A complete non-Randers realization}

\begin{theorem}[Complete non-Randers realization of both transfer problems]\label{thm:global-example}
Define
\begin{equation}\label{eq:Fexample}
 F(x,y)=\alpha(x,y)\Phi_*(q,s),
 \qquad
 q=q_*(f),
 \qquad
 s=\frac{\beta(x,y)}{\alpha(x,y)}.
\end{equation}
Then the following hold.

\begin{enumerate}[label=(\roman*)]
\item $F$ is smooth on $T\R^n\setminus0$, strongly convex, and forward and backward complete.

\item $F$ is non-Riemannian at every $x\neq0$, is not a Randers metric, and is not locally Minkowski.

\item Every outward centered sphere satisfies the exact induced-volume identity
\begin{equation}\label{eq:spherical-exact}
 dV_n^F=dV_{\nbar}^\alpha.
\end{equation}

\item Let
\begin{equation}\label{eq:rhob}
 \rho_b(t)=\rho_*(q_*(t)),
 \qquad
 \rho_*(q)=A_*(q)^n\rho_s(q).
\end{equation}
On the regular set $\R^n\setminus\{0\}$,
\begin{align}
 \alpha^*(df)&=a_\alpha(f)=2\sqrt f,\label{eq:aalpha}\\
 F^*(df)&=a_F(f)=\frac{2\sqrt f}{\rho_b(f)},\label{eq:aF}\\
 \nabla^F f&=\frac{2x}{\rho_b(f)^2},\label{eq:example-grad}\\
 \Delta_F f&=\frac{2n}{\rho_b(f)^2}
 -\frac{4f\rho_b'(f)}{\rho_b(f)^3},\label{eq:example-lap}\\
 H_F&=-\frac{n-1}{\sqrt f\,\rho_b(f)}
 \quad\text{on }M_f=\{|x|=\sqrt f\}.\label{eq:example-H}
\end{align}
Moreover,
\begin{equation}\label{eq:psi-C-explicit}
 \psi(f)=\frac1{\rho_b(f)},
 \qquad
 C(f)=\mathcal C_{M_f}
 =\frac{2\sqrt f\,\rho_b'(f)}{\rho_b(f)^2}.
\end{equation}
Hence the centered-sphere foliation satisfies the same-foliation transfer
conditions.

\item The function $f(x)=|x|^2$ is smooth on all of $\R^n$ and is globally
isoparametric for both $\alpha$ and the BH-$F$ structure.  Thus the centered
spheres realize global same-foliation transfer.  For every $t_0>0$, the same
sphere $M_{t_0}$ also realizes global same-hypersurface transfer by taking the
global $F$-transnormal defining function to be $g=f$.
\end{enumerate}
\end{theorem}

\begin{proof}[Proof of Theorem~\ref{thm:global-example}]
We divide the proof into steps in order to record the construction and the global verification explicitly.

\medskip
\noindent\textbf{Step 1: regularity, strong convexity and completeness.}
The inequalities \eqref{eq:P0-pos}--\eqref{eq:P0-D} prove the standard strong-convexity conditions for the seed profile.  Multiplication by the positive base-point factor $A_*(q)$ does not change fiberwise strong convexity.  Smoothness at the origin follows from \eqref{eq:qstar-euclidean-region}: near $f=0$ one has $q_*(f)=c^2f$, while $\chi$ vanishes on a whole neighborhood of $q=0$, so the potentially odd term involving $\sqrt q$ is absent there.

Because $0\le q\le q_\infty<1/4$, all factors in the profile are uniformly bounded above and below by positive constants.  Hence there exist $m,M>0$ such that
\begin{equation}\label{eq:uniform}
 m|y|\le F(x,y)\le M|y|
\end{equation}
for all $(x,y)$, and the same estimate holds with $y$ replaced by $-y$.  Thus the forward and backward Finsler distances are uniformly comparable with Euclidean distance.  Both metric spaces are complete, and the Finsler Hopf--Rinow theorem gives forward and backward geodesic completeness \cite{BaoChernShen}.

\medskip
\noindent\textbf{Step 2: non-Riemannian, non-Randers and non-locally-Minkowski character.}
If $\chi(q)>0$, the term $2\chi(q)s$ makes the norm nonreversible and hence non-Riemannian.  If $\chi(q)=0$ and $x\ne0$, then
\begin{equation}\label{eq:reversible-nonriem}
 A_*^{-2}F_x^2
 =\alpha^2+2\beta^2+\frac{\beta^4}{\alpha^2},
\end{equation}
which is not a quadratic form when $\beta_x\ne0$ and $n\ge2$.  Hence the metric is non-Riemannian there as well.  On this region the norm is reversible.  A reversible Randers norm is Riemannian, so \eqref{eq:reversible-nonriem} also rules out the Randers class.

On the annulus
\begin{equation}\label{eq:annulus-example}
 \frac1{8c}\le |x|\le\frac1{4c},
\end{equation}
one has $q=c^2|x|^2\in[1/64,1/16]$, $k=c$ and $\chi=1$.  There the profile agrees with a normalized square metric, whose pointwise reversibility invariant is
\begin{equation}\label{eq:rev-invariant}
 \mathfrak r_F(x)
 :=\max_{y\ne0}\frac{F(x,-y)}{F(x,y)}
 =\left(\frac{1+\sqrt q}{1-\sqrt q}\right)^2.
\end{equation}
This varies strictly with $|x|$.  A locally Minkowski metric has base-point-independent norm in suitable local coordinates, hence this coordinate-invariant quantity would be locally constant.  Therefore $F$ is not locally Minkowski on the interior of the annulus.

\medskip
\noindent\textbf{Step 3: exact induced hypersurface volume on centered spheres.}
The BH density of $F$ is
\begin{equation}\label{eq:rho-star}
 \rho_*(q)=A_*(q)^n\rho_s(q).
\end{equation}
Along an outward centered sphere, \eqref{eq:radial-data} gives $r=\sqrt q$.  The primal--dual endpoint conversion of Lemma~\ref{lem:primal-dual} gives
\begin{equation}\label{eq:psi-endpoint}
 \psi(q,\sqrt q)=\frac1{A_*(q)T(q)}.
\end{equation}
Consequently
\begin{equation}\label{eq:Lambda-one}
 \Lambda(q,\sqrt q)
 =\rho_*(q)\psi(q,\sqrt q)
 =\frac{A_*(q)^{n-1}\rho_s(q)}{T(q)}=1.
\end{equation}
The general induced-volume identity $dV_n^F=\Lambda dV_{\nbar}^\alpha$ therefore gives \eqref{eq:spherical-exact}.  This equality is imposed only in the outward radial normal direction, not for every possible normal direction, so it does not trigger all-normal rigidity.

\medskip
\noindent\textbf{Step 4: explicit reduced-compatibility data.}
For $f=|x|^2$,
\begin{equation}\label{eq:f-data}
 \alpha^*(df)=|\nabla^\alpha f|=2\sqrt f,
 \qquad
 \Delta_\alpha f=2n,
 \qquad
 q=q_*(f),
 \qquad
 r=\sqrt{q_*(f)},
 \qquad
 B^\top=0.
\end{equation}
Hence the centered-sphere foliation is strongly compatible and therefore satisfies the same-foliation compatibility conditions.

On the centered-sphere family, $\rho=\rho_b(f)$.  By \eqref{eq:Lambda-one},
\[
 1=\Lambda=\rho\psi=\rho_b(f)\psi(f),
\]
so
\begin{equation}\label{eq:psi-explicit}
 \boxed{\;\psi(f)=\frac1{\rho_b(f)}.\;}
\end{equation}
The same-foliation transnormal relation \eqref{eq:aF-reduced} yields
\begin{equation}\label{eq:aF-derived}
 a_F(f)=a_\alpha(f)\psi(f)
 =\frac{2\sqrt f}{\rho_b(f)},
\end{equation}
which is \eqref{eq:aF}.  Squaring gives
\begin{equation}\label{eq:aF-square}
 F^*(df)^2=\frac{4f}{\rho_b(f)^2}.
\end{equation}

The Finsler gradient formula is
\begin{equation}\label{eq:grad-derived}
 \nabla^Ff=\frac{2x}{\rho_b(f)^2}.
\end{equation}
Hence the outward $F$-unit normal selected by $df$ is
\begin{equation}\label{eq:nF}
 n
 =\frac{\nabla^Ff}{a_F(f)}
 =\frac{x}{\sqrt f\,\rho_b(f)}.
\end{equation}
Since $\rho=\rho_b(f)$ and $B^\top=0$,
\[
 \mathcal C_{M_f}
 =n(\log\rho_b(f)).
\]
Furthermore
\begin{equation}\label{eq:nf}
 n(f)=df(n)=a_F(f)=\frac{2\sqrt f}{\rho_b(f)}.
\end{equation}
Therefore
\begin{equation}\label{eq:C-explicit}
 \boxed{\;
 C(f)=\mathcal C_{M_f}
 =\frac{2\sqrt f\,\rho_b'(f)}{\rho_b(f)^2}.
 \;}
\end{equation}
This proves the explicit same-foliation compatibility \eqref{eq:psi-C-explicit} on the entire regular set.

\medskip
\noindent\textbf{Step 5: the mean-curvature and Laplacian formulas.}
The Euclidean outward unit normal is $\nbar=x/|x|$, so with the sign convention $H_\alpha=-\diver_\alpha\nbar$,
\begin{equation}\label{eq:Halpha}
 H_\alpha=-\frac{n-1}{\sqrt f}.
\end{equation}
Using the reduced mean-curvature relation \eqref{eq:HF-reduced-levelII}, together with
\[
 \psi=\rho_b^{-1},
 \qquad
 \psi'=-\frac{\rho_b'}{\rho_b^2},
 \qquad
 a_\alpha=2\sqrt f,
 \qquad
 C=\frac{2\sqrt f\rho_b'}{\rho_b^2},
\]
we obtain
\begin{align*}
 H_F
 &=\frac1{\rho_b}H_\alpha
 -2\sqrt f\left(-\frac{\rho_b'}{\rho_b^2}\right)
 -\frac{2\sqrt f\rho_b'}{\rho_b^2}\\
 &=\frac{H_\alpha}{\rho_b}
 =-\frac{n-1}{\sqrt f\,\rho_b(f)},
\end{align*}
which proves \eqref{eq:example-H}.  Notice that the two correction terms cancel exactly because the induced-area normalization forces $\rho\psi=1$ on this family.

Likewise, substituting
\[
 b_\alpha=2n,
 \qquad
 a_\alpha=2\sqrt f,
 \qquad
 \psi=\rho_b^{-1},
 \qquad
 C=\frac{2\sqrt f\rho_b'}{\rho_b^2}
\]
into \eqref{eq:DeltaF-levelII} gives
\begin{align*}
 \Delta_F f
 &=\frac{2n}{\rho_b^2}
 +2(4f)\frac1{\rho_b}
 \left(-\frac{\rho_b'}{\rho_b^2}\right)
 +\frac{2\sqrt f}{\rho_b}
 \frac{2\sqrt f\rho_b'}{\rho_b^2}\\
 &=\frac{2n}{\rho_b^2}
 -\frac{8f\rho_b'}{\rho_b^3}
 +\frac{4f\rho_b'}{\rho_b^3}\\
 &=\frac{2n}{\rho_b^2}
 -\frac{4f\rho_b'}{\rho_b^3},
\end{align*}
which is \eqref{eq:example-lap}.  Thus both $a_F$ and $\Delta_F f$ depend only on $f$ throughout the regular set, so $f$ is BH-$F$-isoparametric there.

\medskip
\noindent\textbf{Step 6: global isoparametric realization.}
The function
\[
 f(x)=|x|^2
\]
is smooth on all of $\R^n$; in particular it is $C^1$ at its unique critical
point $x=0$ and smooth on the regular set $\R^n\setminus\{0\}$.  On that
regular set,
\[
 \alpha^*(df)=2\sqrt f,
 \qquad
 \Delta_\alpha f=2n,
\]
so $f$ is globally $\alpha$-isoparametric.  By Steps~4 and~5,
\[
 F^*(df)=\frac{2\sqrt f}{\rho_b(f)},
 \qquad
 \Delta_Ff=
 \frac{2n}{\rho_b(f)^2}
 -\frac{4f\rho_b'(f)}{\rho_b(f)^3},
\]
and both right-hand sides depend only on $f$.  Therefore $f$ is a global
BH-$F$-isoparametric function in the sense of
Section~\ref{subsec:isoparametric-prelim}.  The same global function hence
realizes same-foliation transfer.

For any $t_0>0$, the centered sphere
\[
 M_{t_0}=\{|x|^2=t_0\}
\]
is an $\alpha$-isoparametric hypersurface and a regular level of the global
$F$-isoparametric function $g=f$.  Thus it also realizes global
same-hypersurface transfer.  Equivalently, on each regular sphere
\[
 H_F(t)=-\frac{n-1}{\sqrt t\,\rho_b(t)},
\]
which is constant along the leaf.
\end{proof}

\end{document}